\documentclass[psamsfonts]{amsart}

\usepackage{amsthm,amssymb,amsmath,amscd}

\theoremstyle{plain}
\newtheorem{thm}{Theorem}[section]
\newtheorem{lem}[thm]{Lemma}
\newtheorem{prop}[thm]{Proposition}
\newtheorem{cor}[thm]{Corollary}

\theoremstyle{definition}
\newtheorem*{defn}{Definition}
\newtheorem*{ex}{Example}

\theoremstyle{remark}
\newtheorem*{rem}{Remark}
\newtheorem*{ack}{Acknowledgement}

\newcommand{\newoperator}[2]{\DeclareMathOperator{#1}{#2}}
\newoperator{\spn}{span}
\newoperator{\graph}{graph}
\newoperator{\im}{Im}
\newoperator{\a.e.}{a.e.}
\newoperator{\diag}{diag}
\newoperator{\Aut}{Aut}

\title[Toeplitz-composition $C^{*}$-algebras]
{Quotient algebras of Toeplitz-composition $C^{*}$-algebras
for finite Blaschke products}
\author{Hiroyasu Hamada}

\address{Graduate School of Mathematics, Kyushu University,
Motooka, Fukuoka, 819-0395, Japan.}
\email{h-hamada@math.kyushu-u.ac.jp}



\begin{document}

\begin{abstract}
Let $R$ be a finite Blaschke product.
We study the $C^*$-algebra $\mathcal{TC}_R$
generated by both the composition operator $C_R$
and the Toeplitz operator $T_z$ on
the Hardy space. 
We show that the simplicity of the quotient algebra $\mathcal{OC}_R$
by the ideal of the compact operators can be characterized by
the dynamics near the
Denjoy-Wolff point of $R$
if the degree of $R$ is at least two.
Moreover we prove that the degree of finite Blaschke products is
a complete isomorphism invariant for the class of $\mathcal{OC}_R$
such that $R$ is a finite Blaschke product of degree
at least two and the Julia set of $R$ is the unit circle,
using the Kirchberg-Phillips classification theorem.
\end{abstract}

\maketitle

\section{Introduction}

Let $\mathbb{D}$ be the open unit disk in the complex plane and
$H^2 (\mathbb{D})$ the Hardy space of
analytic functions whose power series have square-summable
coefficients.
For an analytic self-map
$\varphi:\mathbb{D} \rightarrow \mathbb{D} $,
the composition operator
$C_{\varphi} :H^2 (\mathbb{D}) \rightarrow H^2 (\mathbb{D})$
is defined by $C_{\varphi} g = g \circ \varphi$ for $g \in H^2 (\mathbb{D})$
and known to be a bounded operator by the Littlewood subordination theorem
\cite{L}.
The study of composition operators
on the Hardy space $H^2 (\mathbb{D})$ gives a fruitful interplay
between complex analysis and operator theory as
shown, for example, in the books of
Cowen and MacCluer \cite{CM}, and Shapiro \cite{Sha}.  

Let $L^2(\mathbb{T})$ denote the square integrable measurable functions on
the unit circle $\mathbb{T}$ with respect to the normalized Lebesgue measure.
The Hardy space $H^2(\mathbb{T})$ is the closed subspace of
$L^2(\mathbb{T})$ consisting of the functions whose negative
Fourier coefficients vanish.
Let  $P_{H^2} : L^2 (\mathbb{T}) \rightarrow H^2 (\mathbb{T}) \subset L^2 (\mathbb{T})$ be the projection.
For $a \in L ^{\infty} (\mathbb{T})$,
the Toeplitz operator $T_a : H^2(\mathbb{T}) \to H^2(\mathbb{T})$
is defined by $T_a = P_{H^2} M_a$,
where $M_a$ is the multiplication operator by $a$ on $L^2(\mathbb{T})$.
We recall that the $C^*$-algebra $\mathcal{T}$ generated by
the Toeplitz operator $T_z$ contains all continuous symbol Toeplitz
operators, and  its quotient by the ideal
of the compact operators
on $H^2 (\mathbb{T})$ is isomorphic to the commutative $C^*$-algebra
$C(\mathbb{T})$ of all continuous functions on the unit circle $\mathbb{T}$.
Recently several authors considered $C^*$-algebras generated by
composition operators (and Toeplitz operators).
Most of their studies have focused on composition operators induced by
linear fractional maps, that is, rational functions of degree one.
For an analytic self-map
$\varphi:\mathbb{D} \rightarrow \mathbb{D} $,
we denote by
$\mathcal{TC}_\varphi$ the Toeplitz-composition $C^*$-algebra
generated by both the composition operator $C_{\varphi}$ and the Toeplitz
operator $T_z$. Its quotient algebra  by the ideal
$\mathcal{K}$  of the compact operators is denoted by
$\mathcal{OC} _{\varphi}$.
Kriete, MacCluer and Moorhouse
\cite{KMM1, KMM2}
studied the Toeplitz-composition $C^*$-algebra $\mathcal{TC}_\varphi$
for a certain linear fractional self-map $\varphi$.
They describe the quotient algebra $\mathcal{OC} _{\varphi}$
concretely as a subalgebra of $C(\Lambda) \otimes M_2(\mathbb C)$
for a compact space $\Lambda$.
They also considered the $C^*$-algebra generated by $T_z$ and
a finite collection $\{C_{\varphi_1}, \dots, C_{\varphi_n}\}$ of
composition operators whose symbols are such linear fractional self-maps
\cite{KMM3}.
If $\varphi(z) = e^{-2\pi i\theta}z$ for some irrational
number $\theta$, then the Toelitz-composition $C^*$-algebra
$\mathcal{TC}_{\varphi}$ is an extension of the irrational rotation algebra
$A_{\theta}$ by $\mathcal{K}$ and studied in Park \cite{Pa}.
Jury \cite{J2} investigated the $C^*$-algebra generated by
a group of composition operators with the symbols belonging to a
non-elementary Fuchsian group $\Gamma$ to relate it with
extensions of the crossed product $C(\mathbb{T}) \rtimes \Gamma$
by $\mathcal{K}$.
He applied the same analysis to the $C^*$-algebra
$\mathcal{OC}_\varphi$ for a M\"{o}bius transformation $\varphi$
and had the following theorem.

\begin{thm}[Jury \cite{J1}] \label{thm:Jury}
Let $\varphi$ be a M\"{o}bius transformation of $\mathbb{D}$
and let $\alpha$ be an automorphism such that
$\alpha(a) = a \circ{\varphi}$ for $a \in C(\mathbb{T})$.
\begin{enumerate}
\item
If $\varphi$ has finite order $q$, $\mathcal{OC}_\varphi$ is
isomorphic to $C(\mathbb{T}) \rtimes_\alpha \mathbb{Z} / q \mathbb{Z}$.
\item
Otherwise, 
$\mathcal{OC}_\varphi$ is isomorphic to
$C(\mathbb{T}) \rtimes_\alpha \mathbb{Z}$.
\end{enumerate}
\end{thm}

For $0 < s < 1$, let $\varphi_s(z) = sz + (1-s)$. Quertermous \cite{Q}
considered the $C^*$-algebra generated by a semigroup $\{ C_{\varphi_s} \, | \,
0 < s < 1 \}$ and had an exact sequence by the commutator ideal.

It seems that one of easy cases except linear fractional maps
is a rational function of degree two.
We consider the following four
rational functions of degree two:
\begin{align*}
   R_1(z)= \frac{2z^2-1}{2-z^2},& \quad R_2(z)= \frac{2z^2+1}{2+z^2}, \\
   R_3(z)= \frac{3z^2+1}{3+z^2},& \quad R_4(z)= \frac{(3+i)z^2+(1-i)}{(3-i)+(1+i)z^2}.
\end{align*}
Their rational functions resemble each other,
but quotient algebras $\mathcal{OC}_{R_1}, \dots,
\mathcal{OC}_{R_4}$ of Toeplitz-composition $C^*$-algebras
by the ideal $\mathcal{K}$ have a different property.
Quotient algebras $\mathcal{OC}_{R_1}$ and $\mathcal{OC}_{R_3}$
are simple, while
quotient algebras $\mathcal{OC}_{R_2}$ and $\mathcal{OC}_{R_4}$
are not simple.
This is related to a property of complex dynamical systems
$\{ R_j ^{\circ{m}} \} _{m=1} ^{\infty}$ for $j =1, \dots, 4$.
Rational functions $R_1, \dots, R_4$ are finite Blaschke products.

In this paper we consider the quotient algebra $\mathcal{OC}_{R}$ for
a general finite Blaschke product $R$.
We show that the simplicity of the $C^*$-algebra $\mathcal{OC}_R$
can be characterized by the dynamics near the Denjoy-Wolff point of $R$
if the degree of $R$ is at least two.
We note that Watatani and the author \cite{HW}
studied the quotient algebra $\mathcal{OC}_{R}$
when $R$ is finite Blaschke product of degree at least two with
$R(0) =0$.
We showed that
the $C^*$-algebra $\mathcal{OC}_{R}$
is isomorphic to the $C^*$-algebra $\mathcal{O}_{R}(J_R)$
associate with complex dynamical systems introduced by \cite{KW}, which
is simple and purely infinite.

On the other hand, Courtney, Muhly and Schmidt \cite{CMS} studied certain
endomorphisms of $\mathcal{B}(L^2(\mathbb{T}))$ and
$\mathcal{B}(H^2(\mathbb{T}))$,
where $\mathcal{B}(L^2(\mathbb{T}))$ and $\mathcal{B}(H^2(\mathbb{T}))$ are
the $C^*$-algebra of all bounded operators on $L^2(\mathbb{T})$
and $H^2(\mathbb{T})$ respectively.
Let $\varphi$ be an inner function and let $\alpha$ be the induced
endomorphism of
$L^{\infty} (\mathbb{T})$ defined by
$\alpha(a) = a \circ \varphi$ for $a \in L^{\infty} (\mathbb{T})$.
They considered an endomorphism $\beta$
of $\mathcal{B}(L^2(\mathbb{T}))$ satisfying
$\beta(M_a) = M_{\alpha (a)}$  for $a \in L^{\infty} (\mathbb{T})$
and an endomorphism $\beta_{+}$ of
$\mathcal{B}(H^2(\mathbb{T}))$ satisfying
$\beta_{+} (T_a) = T_{\alpha (a)}$ for $a \in L^{\infty} (\mathbb{T})$.
They described such endomorphisms on $\mathcal{B}(H^2 (\mathbb{T}))$
using orthonormal bases of the Hilbert space
$H^2(\mathbb{T}) \ominus T_\varphi H^2 (\mathbb{T})$.
Moreover, when $\varphi = R$ is a finite Blaschke product,
they considered a Hilbert bimodule $L^\infty (\mathbb{T})_{\mathcal{L}}$
over $L^\infty (\mathbb{T})$ and studied
endomorphisms of $\mathcal{B}(L^2(\mathbb{T}))$ satisfying the above
condition.

Fortheremore,
Jury \cite{J1} showed that, if $\varphi$ is a M\"{o}bius transformation
of $\mathbb{D}$, then the quotient algebra
$\mathcal{OC}_\varphi$ is isomorphic to the crossed product by the integer
group or the cyclic group as in Theorem \ref{thm:Jury}.
Let $R$ be a finite Blaschke product of degree at least two with
$R(0) =0$.
Watatani and the author \cite{HW}
proved that the quotient algebra $\mathcal{OC}_{R}$
is isomorphic to the $C^*$-algebra $\mathcal{O}_R(J_R)$ associated
with the complex dynamical system introduced in \cite{KW}.
The $C^*$-algebra $\mathcal{O}_R(J_R)$ is defined as a Cuntz-Pimsner
algebra \cite{Pi}.
In this paper we extend both of these to the case for
a general finite Blaschke product $R$.
We show that the $C^*$-algebra $\mathcal{OC}_{R}$ is isomorphic to
the crossed product $C(\mathbb{T}) \rtimes_\alpha \mathbb{Z} / q \mathbb{Z}$
in Theorem \ref{thm:Jury} or a Cuntz-Pimsner algebra $\mathcal{O}_{X_R}$
associated to a Hilbert bimodule $X_R$ over $C(\mathbb{T})$.

In the proof of this theorem,
one of the keys is to analyze operators of the
form $C_R ^* T_a C_R$ for $a \in L^{\infty}(\mathbb{T})$.
Watatani and the author \cite{HW} showed that,
if $R$ is a finite Blaschke product of degree at least two with
$R(0) =0$, then
the operator $C_R ^* T_a C_R$ is a Toeplitz operator $T_{\mathcal{L}_R(a)}$.
Courtney, Muhly and Schmidt \cite{CMS} extend this
to the case for a general finite Blaschke product.
On the other hand, Jury \cite{J} independently proved a covariant relation
$C_\varphi ^* T_a C_\varphi = T_{A_\varphi (a)}$
for an inner function $\varphi$,
where $A_\varphi$ is the Aleksandrov operator
defined by Aleksandrov-Clark measures \cite{A, Cl}.
If $\varphi = R$ is a finite Blaschke product degree at least two with
$R(0)=0$, then the Aleksandrov
operator $A_R$ is equal to $\mathcal{L}_R$ defined in \cite{HW}.
More generally, Jury also analyzed
the operator $C_\varphi ^* T_a C_\varphi$ for
any analytic self-map on $\mathbb{D}$.
Let $\alpha : C(\mathbb{T}) \to  C(\mathbb{T})$
be the endomorphism $\alpha$ induced by a finite Blaschke product $R$
such that $\alpha(a) = a \circ R$ for $a \in C(\mathbb{T})$.
We should remark that $A_R$ is a transfer operator for
the pair $(C(\mathbb{T}), \alpha)$ in the sense of Exel in \cite{E}.

As noted above, Jury obtained
the covariant relation
$C_\varphi ^* T_a C_\varphi = T_{A_\varphi (a)}$
for an inner function $\varphi$ and $a \in L^{\infty}(\mathbb{T})$.
Exel and Vershik \cite{EV} considered
similar covariant relations on $L^2$ spaces.
Let $(\Omega, \mu)$ be a measure space and let $T$ be a measure-preserving
transformation of $(\Omega, \mu)$.
We denote $S$ by the composition operator on
$L^2 (\Omega, \mu)$ induced by $T$.
They considered a covariant relation
$S ^* M_a S = M_{\mathcal{L}(a)}$ for
$a \in L^\infty (\Omega, \mu)$ under some conditions.

We discuss the difference between the Hilbert bimodule $X_R$ and
the Hilbert bimodule $L^\infty (\mathbb{T})_{\mathcal{L}}$
considered in \cite{CMS}.
Two  Hilbert bimodules are almost the same.
But the inner product of $X_R$ is slightly different from
that of $L^\infty (\mathbb{T})_{\mathcal{L}}$.
The inner product of $X_R$
is given by a weighted sum defined by the weight function
$\frac{1}{|R'|}$ and
naturally comes from the covariant relation
$C_R ^* T_a C_R = T_{A_R (a)}$
for $a \in L^{\infty} (\mathbb{T})$.
However the inner product of $L^\infty (\mathbb{T})_{\mathcal{L}}$
is given by a weighted sum defined by the weight function
$\frac{1}{n}$, where $n$ is the degree
of the finite Blaschke product $R$.
Thanks to the definition of the inner
product of $X_R$, we can prove that
the Hilbert bimodule $X_R$ has a finite orthonormal basis
$\{u_i \}_{i=1} ^n$ which is the first $n$-th
functions of the Takenaka-Malmquist basis \cite{M, T} of
$H^2(\mathbb{T})$.
The Takenaka-Malmquist basis
is known as an orthonormal basis of $L^2(\mathbb{T})$ or
$H^2(\mathbb{T})$.

We also compute the K-group of the $C^*$-algebra $\mathcal{OC}_R$,
using the six term exact sequence obtained by
Pimsner \cite{Pi} and the finite orthonormal basis $\{u_i \}_{i=1} ^n$
of the Hilbert bimodule $X_R$.
Moreover we apply
the Kirchberg-Phillips classification theorem \cite{Ki, Ph} to
the $C^*$-algebra $\mathcal{OC}_R$.
We show that the degree of finite Blaschke products is
a complete isomorphism invariant for the class of $\mathcal{OC}_R$
such that $R$ is a finite Blaschke product of degree
at least two and the Julia set of $R$ is the unit circle $\mathbb{T}$.

%%%%%%%%%%%%%%%%%%%%%%%%%%%%%%%%%%%%%%%%%%%%%%%%%%%%%%%%%%%%%%%%%%%%%%%%%%
%%%%%%%%%%%%%%%%%%%%%%%%%%%%%%%%%%%%%%%%%%%%%%%%%%%%%%%%%%%%%%%%%%%%%%%%%%

\section{Preliminaries}
\subsection{Toeplitz-composition $C^*$-algebras}

Let $L^2(\mathbb{T})$ denote the square integrable measurable functions on
$\mathbb{T}$ with respect to the normalized Lebesgue measure $m$.
The Hardy space $H^2(\mathbb{T})$ is the closed subspace of 
$L^2(\mathbb{T})$ consisting of the functions whose negative 
Fourier coefficients vanish. We put $H^{\infty} (\mathbb{T}) :=
H^2(\mathbb{T}) \cap L^{\infty} (\mathbb{T})$. 

The Hardy space $H^2(\mathbb{D})$ is the Hilbert space consisting
of all analytic  functions $g(z) = \sum_{k=0} ^{\infty} c_k z^ k$
on the open unit disk $\mathbb{D}$ such that
$\sum_{k=0} ^{\infty} | c_k | ^2 < \infty$.
The inner product is given  by
\[
 \langle g, h \rangle = \sum_{k=0} ^{\infty} c_k \overline{d_k} \quad
\text{for} \quad
g(z)= \sum_{k=0} ^{\infty} c_k z^k \ \ \text{and} \ \ 
h(z) = \sum_{k=0} ^{\infty} d_k z^k.
\]

We can identify $H^2 (\mathbb{D})$ with $H^2(\mathbb{T})$ by the
following way.
For $g \in H^2 (\mathbb{D})$, there exists 
\[
 \widetilde{g}(e^{i \theta}) :=\lim_{r \to 1-} g(re^{i \theta})  \ \
 \a.e. \theta
\]
and $\widetilde{g} \in H^2(\mathbb{T})$ by Fatou's theorem.
The inverse $P[f]$ is given as the Poisson integral 
\[
 P[f](z) := \int_{\mathbb{T}} P_z (\zeta) f(\zeta) \ dm(\zeta), \quad z \in \mathbb{D}
\]
for $f \in H^2 (\mathbb{T})$, where $P_z$ is the Poisson kernel defined by
\[
 P_z (\zeta) = \frac{1-|z|^2}{|\zeta - z|^2}, \quad z \in \mathbb{D},
 \ \zeta \in \mathbb{T}.
\]
In this paper we identify $H^2 (\mathbb{T})$ with $H^2(\mathbb{D})$
and use the same notation $H^2$ if no confusion can arise.
An analytic self-map $\varphi$ is called {\it inner}
if $|\widetilde{\varphi} (e^{i \theta}) | = 1$ for almost every $\theta$.

Let  $P_{H^2} : L^2 (\mathbb{T}) \rightarrow H^2 (\mathbb{T}) \subset L^2 (\mathbb{T})$ be the projection.
For $a \in L^{\infty}(\mathbb{T})$,  the Toeplitz operator $T_a$ on
$H^2(\mathbb{T})$ is defined by $T_a f = P_{H^2} a f$
for $f \in H^2 (\mathbb{T})$.
Let $\varphi: \mathbb{D} \rightarrow \mathbb{D}$ be an analytic 
self-map. 
Then the composition operator $C_{\varphi}$ on $H^2 (\mathbb{D})$ is defined by
$C_{\varphi} g = g \circ \varphi$ for $g \in H^2 ({\mathbb{D}})$. By the Littlewood subordination theorem \cite{L}, $C_{\varphi}$ is always bounded.  

We can regard Toeplitz operators and composition operators as acting
on the same Hilbert space.
Put
\begin{align*}
  F(z) = \begin{cases}
          f(z), \quad \quad \ z \in \mathbb{T}, \\
          P[f](z), \quad z \in \mathbb{D},
         \end{cases}
  \quad
  G(z) = \begin{cases}
          \widetilde{g}(z), \quad z \in \mathbb{T}, \\
          g(z), \quad z \in \mathbb{D},
         \end{cases}
\end{align*}
for $f \in H^2(\mathbb{T})$ and $g \in H^2(\mathbb{D})$.
By Ryff \cite[Theorem 2]{Ry}, we know that 
$\widetilde{g \circ \varphi} =
G \circ \widetilde{\varphi}$ for $g \in H^2 (\mathbb{D})$.
If we consider $C_\varphi$ as an operator on $H^2(\mathbb{T})$,
then $C_{\varphi}f = P \widetilde{[f] \circ \varphi}
= F \circ \widetilde{\varphi}$ for $f \in H^2(\mathbb{T})$.
Moreover, if $\varphi$ is an inner function,
we have $C_{\varphi}f = f \circ \widetilde{\varphi}$
for $f \in H^2(\mathbb{T})$,
since $|\widetilde{\varphi} (e^{i \theta}) | = 1$ for almost every $\theta$.

%%%%%%%%%%%%%%%%%%%%%%%%%%%%%%%%%%%%%%%%%%%%%%%%%%%%%%%%%%%%%%%%%%%%%%%%%%

\begin{defn} 
For an analytic self-map
$\varphi:\mathbb{D} \rightarrow \mathbb{D} $, 
we denote by $\mathcal{TC} _{\varphi}$ the $C^*$-algebra generated
by the Toeplitz operator $T_z$ and 
the composition operator  $C_{\varphi}$. 
The $C^*$-algebra $\mathcal{TC} _{\varphi}$ is 
called the {\it Toeplitz-composition $C^*$-algebra} for symbol $\varphi$. 
Since $\mathcal{TC} _{\varphi}$ contains the ideal 
$\mathcal{K}(H^2)$ of compact operators,  
we define a $C^*$-algebra $\mathcal{OC} _{\varphi}$ to be the quotient
$C^*$-algebra $\mathcal{TC} _{\varphi} / \mathcal{K}(H^2)$.
\end{defn}

By abuse of notation, from now on we write $\varphi$ instead of
$\widetilde{\varphi}$.

%%%%%%%%%%%%%%%%%%%%%%%%%%%%%%%%%%%%%%%%%%%%%%%%%%%%%%%%%%%%%%%%%%%%%%%%%%
%%%%%%%%%%%%%%%%%%%%%%%%%%%%%%%%%%%%%%%%%%%%%%%%%%%%%%%%%%%%%%%%%%%%%%%%%%

\subsection{Aleksandrov operators and a covariant relation}

We recall Aleksandrov-Clark measures \cite{A, Cl} and Aleksandrov operators.
The general reference is \cite[Chapter 9]{CMR} for example.

\begin{defn}
Let $\varphi$ be an analytic self-map on $\mathbb{D}$.
For $\alpha \in \mathbb{T}$, there exists a unique measure $\mu_\alpha$
on $\mathbb{T}$ such that
\[
  \frac{1-|\varphi(z)|^2}{|\alpha - \varphi(z)|^2} = \int_\mathbb{T}
  \frac{1-|z|^2}{|\zeta - z|^2} \ d \mu_\alpha(\zeta)
\]
for $z \in \mathbb{D}$ by the Herglotz theorem (see for example
\cite[Theorem 9.1.1]{CMR}).
We call the family $\{\mu_\alpha\}_{\alpha \in \mathbb{T}}$
the family {\it Aleksandrov-Clark measures}
of $\varphi$.
\end{defn}

%%%%%%%%%%%%%%%%%%%%%%%%%%%%%%%%%%%%%%%%%%%%%%%%%%%%%%%%%%%%%%%%%%%%%%%%%%

Each $\mu_\alpha$ has a Lebesgue decomposition
\[
  \mu_\alpha = h_\alpha m + \sigma_\alpha.
\]
The absolutely continuous part of $\mu_\alpha$ is given by
\[
  h_\alpha (\zeta) = \frac{1- |\varphi(\zeta)|^2}{|\alpha - \varphi(\zeta)|^2}
\]
for $\zeta \in \mathbb{T}$.
We note that $\mu_\alpha = \sigma_\alpha$ for $\alpha \in \mathbb{T}$
if $\varphi$ is an inner function.
For a bounded Borel function $a$ on $\mathbb{T}$, we define
\[
  A_\varphi(a)(\alpha) = \int_\mathbb{T} a(\zeta) \ d \mu_\alpha (\zeta)
\]
for $\alpha \in \mathbb{T}$. Then the function $A_\varphi (a)$ is also a
bounded Borel function
and $A_\varphi$ extends to a bounded operator
on $L^{\infty}(\mathbb{T})$.
Moreover $A_\varphi$ is bounded on $C(\mathbb{T})$. 
We call $A_\varphi$ the {\it Aleksandrov operator}.

%%%%%%%%%%%%%%%%%%%%%%%%%%%%%%%%%%%%%%%%%%%%%%%%%%%%%%%%%%%%%%%%%%%%%%%%%%

The following theorem was proved by Jury.

\begin{thm}[Jury {\cite[Corollary 3.4]{J}}] \label{thm:cp_map}
Let $\varphi:\mathbb{D} \to \mathbb{D}$ be an inner function. Then we have
\[
  C_\varphi ^* T_a C_\varphi = T_{A_\varphi(a)}
\]
for $a \in L^{\infty}(\mathbb{T})$.
\end{thm}

More generally,
a similar theorem holds for any analytic self-map on $\mathbb{D}$
 (see \cite{J} for more details).

%%%%%%%%%%%%%%%%%%%%%%%%%%%%%%%%%%%%%%%%%%%%%%%%%%%%%%%%%%%%%%%%%%%%%%%%%%

Let $R$ be a finite Blaschke product on the Riemann sphere
$\widehat{\mathbb{C}} = \mathbb{C} \cup \{\infty \}$,
that is,
\[
 R(z) = \lambda \prod_{k=1} ^n \frac{z-z_k}{1-\overline{z_k}z},
      \quad z \in \widehat{\mathbb{C}},
\]
where $n \in \mathbb{N}, \, \lambda \in \mathbb{T}$ and $z_1,\dots, z_n \in \mathbb{D}$.
Thus $R$ is a rational function of degree $n$. 
Since $R$ is an inner function, $R$ is an analytic self-map on $\mathbb{D}$.
It is known that
$R$ is a finite Blaschke product of degree one if and only if
$R$ is M\"{o}bius transformation of $\mathbb{D}$ .
We note the following fact 
\[
  | R'(\zeta) | = \sum_{k=1} ^n \frac{1- |z_k|^2}{|\zeta-z_k|^2} \neq 0,
  \quad \zeta \in \mathbb{T}.
\]
Thus the finite Blaschke product $R$ has no branched  points on $\mathbb{T}$.
It is known that the Aleksandrov-Clark measures $\mu_\alpha$
of $R$ is given by
\[
   \mu_\alpha = \sum_{\zeta \in R^{-1}(\alpha)} \frac{1}{|R'(\zeta)|}
   \ \delta_\zeta,
\]
where $\delta_\zeta$ is the Dirac measure of
$\zeta$ on $\mathbb{T}$ \cite[Example 9.2.4]{CMR}.

%%%%%%%%%%%%%%%%%%%%%%%%%%%%%%%%%%%%%%%%%%%%%%%%%%%%%%%%%%%%%%%%%%%%%%%%%%
%%%%%%%%%%%%%%%%%%%%%%%%%%%%%%%%%%%%%%%%%%%%%%%%%%%%%%%%%%%%%%%%%%%%%%%%%%

\subsection{Cuntz-Pimsner algebras}

We recall the construction of Cuntz-Pimsner algebras \cite{Pi} (see also
\cite{K}). 
Let $A$ be a $C^*$-algebra and $X$ be a right Hilbert $A$-module.  
We denote by $\mathcal{L}(X)$ the $C^*$-algebra of the adjointable bounded operators 
on $X$.  
For $\xi$, $\eta \in X$, the operator $\theta _{\xi,\eta}$
is defined by $\theta _{\xi,\eta}(\zeta) = \xi \langle \eta, \zeta \rangle_A$
for $\zeta \in X$. 
The closure of the linear span of these operators is denoted by $\mathcal{K}(X)$. 
We say that 
$X$ is a {\it Hilbert bimodule} (or {\it $C^*$-correspondence}) 
over $A$ if $X$ is a right Hilbert $A$-module 
with a homomorphism $\phi : A \rightarrow \mathcal{L}(X)$. We always assume 
that $\phi$ is injective. 

A {\it representation} of the Hilbert bimodule $X$
over $A$ on a $C^*$-algebra $D$
is a pair $(\rho, V)$ constituted by a homomorphism $\rho: A \to D$ and
a linear map $V: X \to D$ satisfying
\[
  \rho(a) V_\xi = V_{\phi(a) \xi}, \quad
  V_\xi ^* V_\eta = \rho( \langle \xi, \eta \rangle_A)
\]
for $a \in A$ and $\xi, \eta \in X$.
It is known that $V_\xi \rho(b) = V_{\xi b}$ follows
automatically (see for example \cite{K}).
We define a homomorphism $\psi_V : \mathcal{K}(X) \to D$
by $\psi_V ( \theta_{\xi, \eta}) = V_{\xi} V_{\eta}^*$ for $\xi, \eta \in X$
(see for example \cite[Lemma 2.2]{KPW}).
A representation $(\rho, V)$ is said to be {\it covariant} if
$\rho(a) = \psi_V(\phi(a))$ for all $a \in J(X)
:= \phi ^{-1} (\mathcal{K}(X))$.

We call a finite set $\{ u_i \}_{i=1} ^n \subset X$ a {\it finite basis}
of $X$ if
$\xi = \sum_{i=1} ^n u_i \langle u_i, \xi \rangle_A$ for any $\xi \in X$.
Let $A$ be unital.
A finite basis $\{u_i \}_{i=1} ^n$ is said to be {\it orthonormal} if
it satisfies $\langle u_i, u_j \rangle_A = \delta_{ij} 1$
for $i,j= 1, \dots, n$. 
Let $(\rho, V)$ be a representation of $X$ over a unital $C^*$-algebra $A$.
Suppose that $\phi$ and $\rho$ is unital,
and that $X$ has a finite basis $\{u_i\}_{i=1}^n$.
Then the representation $(\rho, V)$ is covariant if and only if
$V$ satisfies $\sum_{i=1} ^n V_{u_i} V_{u_i} ^* = 1$.

Let $(i, S)$ be the representation of $X$ which is universal for
all covariant representations. 
The {\it Cuntz-Pimsner algebra} ${\mathcal O}_X$ is 
the $C^*$-algebra generated by $i(a)$ with $a \in A$ and 
$S_{\xi}$ with $\xi \in X$.
We note that $i$ is known to be injective
\cite{Pi} (see also \cite[Proposition 4.11]{K}).
We usually identify $i(a)$ with $a$ in $A$.
There exists an action $\gamma: \mathbb{T} \to \Aut \mathcal{O}_X$
with $\gamma_{\lambda} (S_\xi) = \lambda S_\xi $ for $\lambda \in \mathbb{T}$
and $\xi \in X$,
which is called the {\it gauge action}.
We denote by $\mathcal{O}_X ^{\mathbb{T}}$ the fixed point algebra of $\gamma$.
We define a faithful conditional expectation of $\mathcal{O}_X$ onto
$\mathcal{O}_X ^{\mathbb{T}}$ by
$E(T) = \int_{\mathbb{T}} \ \gamma_{\lambda} (T) \ dm(\lambda)$ for
$T \in \mathcal{O}_X$.

\begin{lem} \label{lem:rep}
Let $A$ and $D$ be unital $C^*$-algebras and
let $X$ be a right Hilbert $A$-module.
Suppose $\rho: A \to D$ is a unital injective homomorphism, $V: X \to D$
is a linear map and $\{ u_i \}_{i=1} ^n \subset X$ satisfying
\[
  V_\xi \rho (a) = V_{\xi a}, \quad
  V_{\xi} ^* V_{\eta} = \rho(\langle \xi, \eta \rangle_A), \quad
  \sum _{i=1} ^n V_{u_i} V_{u_i} ^* =1
\]
for $a \in A$ and $\xi, \eta \in X$.
Then $V$ is injective and $\{u_i \}_{i=1} ^n$ is a finite basis of X.
Moreover, if
$\{ u_i \}_{i=1} ^n$ satisfies $V_{u_i} ^* V_{u_j} = \delta_{ij} 1$,
then $\{u_i \}_{i=1} ^n$ is a finite orthonormal basis of $X$.
\end{lem}

\begin{proof}
For $\xi \in X$,
\[
  V_\xi = \sum_{i=1} ^n V_{u_i} V_{u_i} ^* V_\xi
  = \sum_{i=1} ^n V_{u_i} \rho ( \langle u_i, \xi \rangle_A)
  = V_{\sum_{i=1} ^n u_i \langle u_i, \xi \rangle_A}.
\]
Since $\rho$ is injective, $V$ is also injective
(see for example \cite[Lemma 2.2]{KPW}).
Thus we have $\xi = \sum_{i=1} ^n u_i \langle u_i, \xi \rangle_A$.
This implies that $\{u_i \}_{i=1} ^n$ is a finite basis of X.
Suppose $\{ u_i \}_{i=1} ^n$
satisfies $V_{u_i} ^* V_{u_j} = \delta_{ij} 1$.
Since $\rho( \langle u_i, u_j \rangle _A) = \delta_{ij} 1$,
we have $\langle u_i, u_j \rangle _A = \delta_{ij} 1$.
\end{proof}

\subsection{Complex dynamical systems}

Let $R$ be a rational function of degree at least two. 
The sequence $\{ R^{\circ m} \}_{m=1} ^{\infty}$ of
iterations of composition by $R$ 
gives a complex dynamical system on the Riemann sphere
$\widehat{\mathbb C}$. 
The {\it Fatou set} $F_R$ of $R$ is the maximal open subset of 
$\widehat{\mathbb C}$ on which $\{ R^{\circ m} \}_{m=1} ^{\infty}$
is equicontinuous (or a normal family), and the {\it Julia set}
$J_R$ of $R$ is the 
complement of the Fatou set in $\widehat{\mathbb C}$. 
It is known that the Julia set $J_R$ is not empty.

We recall the classification of fixed points of
an analytic function $f$.
Let $w_0 \in \mathbb{C}$ be a fixed point of $f$.
The number $| f'(w_0) |$ is called the {\it multiplier} of $f$ at $w_0$.
\begin{enumerate}
\item When $| f'(w_0) | < 1$, we call $w_0$ an {\it attracting} fixed point.
\item When $| f'(w_0) | > 1$, we call $w_0$ a {\it repelling} fixed point. 
\item When $f'(w_0)$ is a root of unity, we call $w_0$ a {\it parabolic}
(or {\it rationally indifferent}, or {\it rationally neutral})
fixed point.
\item When $|f'(w_0)|=1$, but $f'(w_0)$ is not a root of unity,
we call $w_0$ an {\it irrationally indifferent}
(or {\it irrationally neutral}) fixed point.
\end{enumerate}

%%%%%%%%%%%%%%%%%%%%%%%%%%%%%%%%%%%%%%%%%%%%%%%%%%%%%%%%%%%%%%%%%%%%%%%%%%%%%

The local dynamics near a parabolic fixed point of an analytic function is
described by the Leau-Fatou flower theorem.
We may consider the case that $f$ has a parabolic fixed point at
$0$ with multiplier $1$.
Let $f$ be an anaytic function represented by a convergent power series
\[
   f(z) = z + c_{p+1} z^{p+1} + \cdots \quad (c_{p+1} \neq 0)
\]
at 0.
An {\it attracting petal} is a simply connected open set $U$ such that
$0 \in \partial U$, $f(U) \subset U$ and $f^{\circ m}(z) \to 0$ as
$m \to \infty$ for all $z \in U$.
A {\it repelling petal} is an attracting petal for $f^{-1}$.

%%%%%%%%%%%%%%%%%%%%%%%%%%%%%%%%%%%%%%%%%%%%%%%%%%%%%%%%%%%%%%%%%%%%%%%%%%%%%

\begin{thm}[Leau-Fatou flower theorem {\cite[Theorem 10.5]{Mi}}]
\label{thm:petal}
Let $f$ be an analytic function represented by a convergent power series
\[
   f(z) = z + c_{p+1} z^{p+1} + \cdots \quad (c_{p+1} \neq 0)
\]
at 0. Then there exist attracting petals $U_1, \dots , U_p$ and
repelling petal $V_1, \dots, V_p$ such that
\begin{enumerate}
\item
$U_j \cap U_k = \emptyset$ and $V_j \cap V_k = \emptyset$ for $j \neq k$.
\item
The union $\bigcup _j ^p U_j \cup \bigcup _j ^p V_j \cup \{0\}$ is an open
neighborhood of $0$.
\end{enumerate}
\end{thm}

%%%%%%%%%%%%%%%%%%%%%%%%%%%%%%%%%%%%%%%%%%%%%%%%%%%%%%%%%%%%%%%%%%%%%%%%%%%%%

It is known that the Julia set of
a finite Blaschke product of degree at least two
can be computed as follows.

\begin{prop}[{\cite[p. 79]{CG}}] \label{prop:classification}
Let $R$ be a finite Blaschke product of degree at least two.
Then one of the following holds.
\begin{enumerate}
\item
The finite Blaschke product $R$ has a fixed point in $\mathbb{D}$.
\item
The finite Blaschke product $R$ has an attracting fixed point on $\mathbb{T}$.
\item
The finite Blaschke product $R$ has a parabolic fixed point on $\mathbb{T}$
with two attracting petals.
\item
The finite Blaschke product $R$ has a parabolic fixed point on $\mathbb{T}$
with one attracting petal.
\end{enumerate}
Moreover, if $R$ satisfies $(1)$ or $(3)$, then the Julia set $J_R$ is
$\mathbb{T}$. Otherwise the Julia set $J_R$ is
a Cantor set on $\mathbb{T}$.
\end{prop}

%%%%%%%%%%%%%%%%%%%%%%%%%%%%%%%%%%%%%%%%%%%%%%%%%%%%%%%%%%%%%%%%%%%%%%%%%%%%%

The above proposition is related to the Denjoy-Wolff Theorem. 

\begin{thm}[Denjoy-Wolff theorem {\cite[Theorem 2.51]{CM}}]
Let $\varphi: \mathbb{D} \to \mathbb{D}$ be analytic, and assume
$\varphi$ is not an elliptic M\"{o}bius transformation nor the identity.
Then there exists $w_0 \in \overline{\mathbb{D}}$ such that
$\{\varphi ^{\circ m}\}_{m=1} ^{\infty}$ converges
to $w_0$ uniformaly on compact subsets of $\mathbb{D}$.
\end{thm}

The limit point of the Denjoy-wolff theorem will be referred to as
the {\it Denjoy-Wolff point} of $\varphi$.

There are some remarks about Proposition \ref{prop:classification}.
Let $R$ be a finite Blaschke product of degree at least two.
It is known that
there exists exactly one fixed point
$w_0 \in \overline{\mathbb{D}}$ of $R$ such that $| R'(w_0) | \leq 1$.
Moreover $w_0$ is the Denjoy-Wolff point of $R$
(see for example \cite[p. 59]{CM}).
Thus the fixed point in Proposition \ref{prop:classification}
is the Denjoy-Wolff point $w_0$ of $R$.
If $w_0 \in \mathbb{D}$, then $w_0$ is automatically
attracting by the Schwarz lemma.
Otherwise, it is known that $0 < R'(w_0) \leq 1$ (see for example
\cite[p. 55 and Theorem 2.48]{CM}).
If $w_0$ is a parabolic fixed point on $\mathbb{T}$, then
$R'(w_0) =1$. Moreover, if $R''(w_0) = 0$,
then $R$ belongs to (3) by Theorem \ref{thm:petal}.
Otherwise, $R$ belongs to (4).

%%%%%%%%%%%%%%%%%%%%%%%%%%%%%%%%%%%%%%%%%%%%%%%%%%%%%%%%%%%%%%%%%%%%%%%%%%
%%%%%%%%%%%%%%%%%%%%%%%%%%%%%%%%%%%%%%%%%%%%%%%%%%%%%%%%%%%%%%%%%%%%%%%%%%

\section{Representations induced by
Toeplitz operators and composition operators}

Let $R$ be a finite Blaschke product.
We shall construct a Hilbert bimodule $X_R$ and its covariant representation
induced by the Toeplitz operator $T_z$ and the composition operator $C_R$.
In Section \ref{sect:main},
we shall prove the $C^*$-algebra $\mathcal{OC}_R$ is
isomorphic to the Cuntz-Pimsner algebra $\mathcal{O}_{X_R}$
associated to the Hilbert bimodule $X_R$
when $R$ is a M\"{o}bius transformation of infinite order which maps $\mathbb{D}$ to itself 
or a finite Blaschke product of degree at least two.

Let $R$ be a finite Blaschke product and $A=X_R=C(\mathbb{T})$.
Then $X_R$ is an $A$-$A$ bimodule over $A$ by
\[
(a \cdot \xi \cdot b)(z) = a(z)\xi(z)b(R(z)), \quad a, b \in A, \ \xi \in X_R.
\]
We define an $A$-valued sesquilinear form
$\langle \, ,\, \rangle_A : X_R \times X_R \to A$ by
\[
  \langle \xi, \eta \rangle _A  = A_R (\overline{\xi} \eta), \quad
  \xi, \eta \in X_R.
%  = \sum _{z \in R^{-1}(w)}
%  \frac{1}{|R'(z)|} \overline{\xi (z)} \eta (z)
\]
Since the Aleksandrov operator $A_R$ is a faithful positive map
on $A$, the sesquilinear form
$\langle \, ,\, \rangle_A$ is an $A$-valued inner product on $X_R$.
Put $\| \xi \|_2 = \| \langle \xi, \xi \rangle_A \|^{1/2}$
for $\xi \in X_R$.
%%%%%%%%%%%%%%%%%%%%%%%%%%%%%%%%%%%%%%%%%%%%%%%%%%%%%%%%%%%%%%%%%%%%%%%%%%

\begin{prop}
Let $R$ be a finite Blaschke product. Then $X_R$ is a full Hilbert bimodule
over $A$ without completion. The left action $\phi: A \to \mathcal{L}(X_R)$
is unital and faithful.
\end{prop}

%%%%%%%%%%%%%%%%%%%%%%%%%%%%%%%%%%%%%%%%%%%%%%%%%%%%%%%%%%%%%%%%%%%%%%%%%%

\begin{proof}
Since $R'$ is continuous on $\mathbb{T}$, there exists $M > 0$ such that
$\frac{1}{|R'(z)|} \leq M$ for $z \in \mathbb{T}$.
For $\xi \in X_R$, we have
\[
   \| \xi \| \leq \| \xi \|_2
   = \left ( \sup_{w \in \mathbb{T}} \sum_{z \in R^{-1}(w)}
   \frac{1}{|R'(z)|} |\xi(z)|^2 \right )^{1/2}
   \leq \sqrt{nM} \| \xi \|.
\]
Thus the two norms $\| \ \|_2$ and $\| \ \|$ are equivalent.
Since $X_R$ is complete with respect to $\| \ \|$, it is also
complete with respect to $\| \ \|_2$.
Since $A_R(1)$ is a positive invertible element in $A$,
$\langle X_R, X_R \rangle_A$
contains the identity of $A$. Thus $X_R$ is full.
It is clear that $\phi$ is faithful.
\end{proof}

%%%%%%%%%%%%%%%%%%%%%%%%%%%%%%%%%%%%%%%%%%%%%%%%%%%%%%%%%%%%%%%%%%%%%%%%%%
Let $R$ be a finite Blaschke product,
that is,
\[
 R(z) = \lambda \prod_{k=1} ^n \frac{z-z_k}{1-\overline{z_k}z},
      \quad z \in \widehat{\mathbb{C}},
\]
where $n \in \mathbb{N}, \, \lambda \in \mathbb{T}$ and $z_1,\dots, z_n \in \mathbb{D}$.
We conisder the following functions
\[
  u_1 (z) = \frac{\sqrt{1-|z_1|^2}}{1-\overline{z_1}z}, \quad
  u_i (z) = \frac{\sqrt{1-|z_i|^2}}{1-\overline{z_i}z}
  \prod_{k=1} ^{i-1} \frac{z - z_k}{1-\overline{z_k}z}, \quad z \in \mathbb{T}
\]
for $i=2, \dots, n$.
It is known that $\{u_i \}_{i=1} ^n$ is an
orthonormal basis of $\mathcal{D}:= H^2 \ominus T_R H^2$ and
$\{u_i R^k \, | \, i= 1, \dots ,n, \ k= 0, 1, \dots \}$ is an
orthonormal basis of $H^2$, where
$R^k$ is the $k$-th power of $R$ with respect to pointwise
multiplication \cite{M, T} (see also \cite[Lemma 3.1, Remark 3.7]{CMS}).
The basis $\{u_i R^k \, | \, i= 1, \dots ,n, \ k= 0, 1, \dots \}$ is called the
{\it Takenaka-Malmquist basis} of $H^2$.

%%%%%%%%%%%%%%%%%%%%%%%%%%%%%%%%%%%%%%%%%%%%%%%%%%%%%%%%%%%%%%%%%%%%%%%%%%

\begin{lem} \label{lem:Cuntz relation}
The notations be as above. Then
\[
  (T_{u_i} C_R)^* T_{u_j} C_R = \delta_{ij} I \quad \text{and} \quad
  \sum_{i=1} ^n T_{u_i} C_R (T_{u_i} C_R)^* = I.
\]
\end{lem}

%%%%%%%%%%%%%%%%%%%%%%%%%%%%%%%%%%%%%%%%%%%%%%%%%%%%%%%%%%%%%%%%%%%%%%%%%%

\begin{proof}
Set $e_k(z) = z^k$. We have
$T_{u_i} C_R e_k = u_i R^{k}$
for $i = 1, \dots ,n$ and $k \geq 0$.
Thus $T_{u_i} C_R$ is an isometry
and $\bigoplus_{i = 1} ^n \im (T_{u_i} C_R) = H^2$.
\end{proof}

%%%%%%%%%%%%%%%%%%%%%%%%%%%%%%%%%%%%%%%%%%%%%%%%%%%%%%%%%%%%%%%%%%%%%%%%%%

\begin{prop} \label{prop:relation}
For $a \in A$ and $\xi \in X_R$,
we define $\rho(a) = \pi(T_a)$ and $V_\xi = \pi(T_{\xi} C_R)$,
where $\pi$ is the canonical quotient map $\mathcal{TC}_R$ to
$\mathcal{OC}_R$.
Then $(\rho, V)$ is a covariant representation of $X_R$ on $\mathcal{OC}_R$
and $\rho$ is unital and injective.
Moreover, $\{u_i \}_{i=1} ^n$ is a finite orthonormal basis of $X_R$.
\end{prop}

%%%%%%%%%%%%%%%%%%%%%%%%%%%%%%%%%%%%%%%%%%%%%%%%%%%%%%%%%%%%%%%%%%%%%%%%%%

\begin{proof}
Let $a \in A$ and $\xi, \eta \in X_R$.
By definition, we have
\[
  \rho(a) V_{\xi} = V_{a \cdot \xi}.
\]
By Theorem \ref{thm:cp_map}, it follows that
\[
  V_{\xi} ^* V_{\eta} = \pi(C_R ^* T_{\overline{\xi} \eta} C_R)
  = \pi(T_{A_R(\overline{\xi} \eta)}) = \rho(\langle \xi, \eta \rangle_A).
\]
Lemma \ref{lem:Cuntz relation} implies
\[
  V_{u_i} ^* V_{u_j} = \pi ((T_{u_i} C_R)^* T_{u_j} C_R) = \delta_{ij} I
\]
for $i,j = 1, \dots, n$ and
\[
  \sum_{i=1} ^n V_{u_i} V_{u_i} ^*
  = \sum_{i=1} ^n \pi (T_{u_i} C_R (T_{u_i} C_R)^*) = I.
\]
It is known that $\rho$ is injective.
The rest is now clear from Lemma \ref{lem:rep}.
\end{proof}

%%%%%%%%%%%%%%%%%%%%%%%%%%%%%%%%%%%%%%%%%%%%%%%%%%%%%%%%%%%%%%%%%%%%%%%%%%

\begin{rem}
Let $R$ be a finite Blaschke product of degree $n$.
We can also show the same statement
for any orthonormal basis $\{v_i \}_{i=1} ^n$ of $\mathcal{D}
= H^2 \ominus T_R H^2$ as above, since
$\{v_i R^k \, | \, i= 1, \dots ,n, \ k= 0, 1, \dots \}$ is an
orthonormal basis of $H^2$ and $v_i \in H^{\infty}$ for $i=1, \dots , n$.
Moreover we may
use $\{v_i \}_{i=1} ^n$ instead of
$\{u_i \}_{i=1} ^n$ from now on, since
$v_i \in C(\mathbb{T})$ for $i=1, \dots , n$.
\end{rem}

%%%%%%%%%%%%%%%%%%%%%%%%%%%%%%%%%%%%%%%%%%%%%%%%%%%%%%%%%%%%%%%%%%%%%%%%%%

\section{Uniqueness}
Let $R$ be a M\"{o}bius transformation of infinite order which maps $\mathbb{D}$ to itself 
or a finite Blaschke product of degree at least two.
We shall show $\mathcal{OC}_R$ is isomorphic to the
Cuntz-Pimsner algebra $\mathcal{O}_{X_R}$.
In this section we consider a more general setting.
Let $A=C(\mathbb{T})$ and $D$ be a unital $C^*$-algebra.
Suppose $(\rho, V)$ is a covariant representation of $X_R$ on $D$
such that $\rho$ is unital and injective.
We denote by $B$  the  $C^*$-algebra generated by
$\{ \rho(a), V_\xi \, | \, a \in A, \xi \in X_R \}$.
We shall show $B$ is isomorphic to $\mathcal{O}_{X_R}$.
This implies that
the Cuntz-Pimsner algebra $\mathcal{O}_{X_R}$ is uniquely determined
by covariant relations.
There have been many studies on uniqueness theorems
by many authors (see for example \cite{KPW,K3}).
In this paper we give a self-contained proof in this case
using the finite orthonormal basis $\{u_i\}_{i=1} ^n$ of $X_R$.
Put $\alpha (a) = a \circ R$ for $a \in A$.

%%%%%%%%%%%%%%%%%%%%%%%%%%%%%%%%%%%%%%%%%%%%%%%%%%%%%%%%%%%%%%%%%%%%%%%%%%

\begin{lem} \label{lem:matrix}
Let $R$ be a M\"{o}bius transformation of infinite order which maps $\mathbb{D}$ to itself 
or a finite Blaschke product of degree at least two.
For any $N \in \mathbb{N}, a \in M_N(C(\mathbb{T})), \varepsilon > 0$
and $m \in \mathbb{N}$,
there exists $c \in A = C(\mathbb{T})$ satisfying the following
\begin{enumerate}
\item $0 \leq c \leq 1$,
\item $c \alpha ^j (c) = 0$ for $j= 1, \dots, m,$
\item $\| a \, \diag(c^2) \| \geq \| a \| - \varepsilon$ in $M_N(C(\mathbb{T}))$,
where $\diag(T)$ is the diagonal matrix whose diagonal elements are all equal
to T.
\end{enumerate}
\end{lem}

%%%%%%%%%%%%%%%%%%%%%%%%%%%%%%%%%%%%%%%%%%%%%%%%%%%%%%%%%%%%%%%%%%%%%%%%%%

\begin{proof}
We may identify $a \in M_N(C(\mathbb{T}))$
with a continuous $M_N (\mathbb{C})$-valued function on
$\mathbb{T}$. For $m \in \mathbb{N}$,
$W_m := \{z \in \mathbb{T} \, | \, R^{\circ j} (z) \neq z \ \text{for any} \,
j = 1, \dots , m\}$ is dense in $\mathbb{T}$,
since $R^{\circ k}$ is a non-identical rational function
for $k \in \mathbb{N}$.
Thus for any $\varepsilon > 0$,
there exists $w_0 \in W_m$ such that
$\| a(w_0) \| \geq \| a \| - \varepsilon$.
Since $R^{\circ j}(w_0) \neq w_0$,
there exist open neighborhoods $U_j$ of $w_0$
such that $U_j \cap (R^{\circ j})^{-1} (U_j) = \emptyset$ for $j= 1, \dots, m$.
Put $U = \bigcap_{j=1} ^m U_j$.
We can choose $c \in A = C(\mathbb{T})$ such that
$c(w_0) = 1, \, 0 \leq c \leq 1$
and $c(z) = 0$ for $z\in U^c$. 
Since $U \cap (R^{\circ j})^{-1} (U) = \emptyset$, it follows that
${\rm supp}(c) \cap {\rm supp} (\alpha^j (c) ) = \emptyset$
for $j= 1, \dots, m$.
Thus we have $c \alpha^j(c) = 0$ for $j = 1, \dots ,m$.
We identify $a \diag(c^2)$ with a continuous $M_N (\mathbb{C})$-valued
function on $\mathbb{T}$. Then
\[
  \| a \diag(c^2) \| \geq \|a(w_0) \diag(c(w_0)^2) \| = \| a(w_0) \|
  \geq \| a \| - \varepsilon.
\]
\end{proof}

%%%%%%%%%%%%%%%%%%%%%%%%%%%%%%%%%%%%%%%%%%%%%%%%%%%%%%%%%%%%%%%%%%%%%%%%%%

Since $(\rho, V)$ is a representation of $X_R$,
we can define the representation
$(\rho,V^{(p)})$ of $X_R ^{\otimes p}$ such that
$V_\xi ^{(p)} = V_{\xi_1} \dots V_{\xi_p}$ for
$\xi = \xi_1 \otimes \dots \otimes \xi_p \in X_R ^{\otimes p}$.
For simplicity notation, we write $V_\xi$ instead of $V_\xi ^{(p)}$
for $\xi \in X_R ^{\otimes p}$ if no confusion can arise.
Let $B_{p,q}$ be the closed linear span of
$\{ V_\xi V_\eta ^*  \, | \,
\xi \in X_R ^{\otimes q}, \eta \in X_R ^{\otimes p} \}$.
We note that $B_{p,p}$ is a $C^*$-subalgebra of $B$.

For a sequence $\mu = (i_1, \dots , i_p) \in \{1, \dots, n \}^p$, set
$u_{\mu} = u_{i_1} \otimes \dots \otimes u_{i_p}$ and $| \mu | = p$.
Since $\{u_\mu \, | \ |\mu| = p \}$ is a finite orthonormal
basis of $X_R ^{\otimes p}$, we have
\[
  V_\xi = \sum_{|\mu|=p} V_{u_{\mu}}
  \rho(\langle u_\mu, \xi \rangle_A)
\]
for $\xi \in X_R ^{\otimes p}$.
We can rewrite $B_{p,p}$ as the closed linear span of
\[
  \{ V_{u_\mu}  \rho(a_{\mu \nu}) V_{u_\nu} ^* \, | \  |\mu|= |\nu| = p, \,
a_{\mu \nu} \in A \}.
\]
It is clear that $\{u_\mu \, | \ |\mu| = p \}$ is a finite orthonormal
basis of $X_R ^{\otimes p}$. Thus
$V_{u_\mu}^* V_{u_\nu} = \delta_{\mu \nu}I$
and $\sum_{|\mu|=p} V_{u_{\mu}} V_{u_{\mu}}^* = 1$.
Since $\rho$ is injective, there exists an isomorphism
$\Psi:B_{p,p} \to M_{n^p}(C(\mathbb{T})) \cong
\bigotimes _{k=1} ^p M_n(C(\mathbb{T}))$
such that
$\Psi(V_{u_\mu}  \rho(a_{\mu \nu}) V_{u_\nu} ^*) = a_{\mu \nu} e_{i_1 j_1}
\otimes \dots \otimes e_{i_p j_p}$,
where $\mu = (i_1, \dots , i_p), \, \mu = (j_1, \dots , j_p)$ and
$\{e_{ij}\}_{i,j = 1} ^n$
is the standard matrix units of $M_n(\mathbb{C})$.

%%%%%%%%%%%%%%%%%%%%%%%%%%%%%%%%%%%%%%%%%%%%%%%%%%%%%%%%%%%%%%%%%%%%%%%%%%

\begin{lem} \label{lem:fixed point}
Let $R$ be a M\"{o}bius transformation of infinite order which maps $\mathbb{D}$ to itself 
or a finite Blaschke product of degree at least two.
For any $p \in \mathbb{N}, T_0 \in B_{p, p}, \varepsilon > 0$
and $m \in \mathbb{N}$,
there exists $c \in A$ satisfying $(1)$, $(2)$ of
Lemma $\ref{lem:matrix}$ and
\[
 \| \rho(\alpha^p(c)) T_0 \rho(\alpha^p(c)) \| \geq \| T_0 \| - \varepsilon.
\]
\end{lem}

%%%%%%%%%%%%%%%%%%%%%%%%%%%%%%%%%%%%%%%%%%%%%%%%%%%%%%%%%%%%%%%%%%%%%%%%%%

\begin{proof}
Take $T_0 \in B_{p, p}$ with
\[
  T_0 = \sum_{|\mu| = |\nu| =p}
  V_{u_\mu} \rho(a_{\mu \nu}) V_{u_\nu} ^*.
\]
Put $b=(b_{ij})_{i,j=1, \dots , n^p} = \Psi(T_0)$.
Choose $c \in A$ as in Lemma \ref{lem:matrix}.
Since $\rho(\alpha^p(c))V_{u_\mu} = V_{u_\mu} \rho (c)$ for $| \mu | = p$,
we have
\begin{align*}
  \| \rho(\alpha^p(c)) T_0 \rho(\alpha^p(c)) \| & =
  \left \|\sum_{|\mu| = |\nu| =p}
  V_{u_\mu} \rho(a_{\mu \nu} c^2) V_{u_\nu} ^* \right \| \\
  &= \| (b_{ij} c^2)_{i,j} \|
  = \| b \diag(c^2) \|
  \geq \| b \| - \varepsilon
  = \| T_0 \| - \varepsilon
\end{align*}
by Lemma \ref{lem:matrix}.
\end{proof}

%%%%%%%%%%%%%%%%%%%%%%%%%%%%%%%%%%%%%%%%%%%%%%%%%%%%%%%%%%%%%%%%%%%%%%%%%%

We denote by $B^{\rm{alg}}$ the $*$-algebra generated by
$\{ \rho(a), V_\xi \, | \, a \in A, \xi \in X_R \}$.

%%%%%%%%%%%%%%%%%%%%%%%%%%%%%%%%%%%%%%%%%%%%%%%%%%%%%%%%%%%%%%%%%%%%%%%%%%

\begin{lem} \label{lem:free}
Let $R$ be a M\"{o}bius transformation of infinite order which maps $\mathbb{D}$ to itself 
or a finite Blaschke product of degree at least two
and let $p, m \in \mathbb{N}$.
Suppose $T \in B^{\rm{alg}}$ satisfies $T = \sum_{j=-m} ^m T_j$ with
$T_j \in B_{p,p+j}$.
For any $\varepsilon > 0$, there exists $d \in A$ satisfying
the following
\begin{enumerate}
\item $0 \leq d \leq 1$,
\item $\rho(d) T_j \rho(d) = 0$ for $j \neq 0$,
\item $\| \rho(d) T_0 \rho(d) \| \geq \| T_0 \| - \varepsilon$.
\end{enumerate}
\end{lem}

%%%%%%%%%%%%%%%%%%%%%%%%%%%%%%%%%%%%%%%%%%%%%%%%%%%%%%%%%%%%%%%%%%%%%%%%%%

\begin{proof}
We choose $c \in A$ as in Lemma \ref{lem:fixed point}
and put $d = \alpha^p(c)$.
It is clear that $d$ satisfies (1) and (3) by Lemma \ref{lem:fixed point}.
For $-m \leq j \leq m$, $T_j$ is a finite sum of terms in the
form such that
\[ V_\xi V_\eta ^*, \quad \xi \in X_R ^{\otimes p+j},
   \ \eta \in X_R ^{\otimes p}.
\]
For $1 \leq j \leq m$, we have
\begin{align*}
  \rho(d) V_\xi V_\eta ^* \rho(d)
  &= \rho(\alpha^p (c)) V_\xi V_\eta ^* \rho(\alpha^p(c))
  = \rho(\alpha^p (c)) V_\xi \rho(c) V_\eta ^* \\
  &= \rho(\alpha^p (c)) \rho(\alpha^{p+j} (c)) V_\xi V_\eta ^*
  = \rho(\alpha ^p (c \alpha^j (c))) V_\xi  V_\eta ^* = 0.
\end{align*}
Thus $\rho(d) T_j \rho(d) = 0$. The proof of the case
$-m \leq j \leq -1$ is similar.
\end{proof}

%%%%%%%%%%%%%%%%%%%%%%%%%%%%%%%%%%%%%%%%%%%%%%%%%%%%%%%%%%%%%%%%%%%%%%%%%%

We now show the following uniquness theorem.

%%%%%%%%%%%%%%%%%%%%%%%%%%%%%%%%%%%%%%%%%%%%%%%%%%%%%%%%%%%%%%%%%%%%%%%%%%

\begin{prop} \label{prop:uniqueness}
Let $R$ be a M\"{o}bius transformation of infinite order which maps $\mathbb{D}$ to itself 
or a finite Blaschke product of degree at least two.
Let $A=C(\mathbb{T})$ and $D$ be a unital $C^*$-algebra.
Suppose $(\rho, V)$ is a covariant representation of
the Hilbert bimodule $X_R$ on $D$ such that $\rho$ is unital and injective
and $B$ is a $C^*$-algebra generated by $\rho(a)$
with $a \in A$ and $V_\xi$ with $\xi \in X_R$.
Then there exists an isomorphism $\Phi: B \to \mathcal{O}_{X_R}$
such that $\Phi(\rho(a)) = a$ and $\Phi(V_\xi) =S_\xi$ for $a \in A, \,
\xi \in X_R$.
\end{prop}

%%%%%%%%%%%%%%%%%%%%%%%%%%%%%%%%%%%%%%%%%%%%%%%%%%%%%%%%%%%%%%%%%%%%%%%%%%

\begin{proof}
By the universality of $\mathcal{O}_{X_R}$,
there exists a surjective homomorphism $\Psi: \mathcal{O}_{X_R} \to B$
such that
$\Psi(a) = \rho(a)$ and $\Psi(S_\xi) = V_{\xi}$ for $a \in A$ and
$\xi \in X_R$.
Let $T \in B^{alg}$.
There exist $p, m \in \mathbb{N}$ such that $T=\sum_{j=-m} ^m T_j$
and $T_j \in B_{p,p+j}$.
For any $\varepsilon > 0$, we choose $d \in A$ as
in Lemma \ref{lem:free}.
We have
\[
  \| T \| \geq \| \rho(d) T \rho(d) \|
  = \| \rho(d) T_0  \rho(d) \| \geq \| T_0 \| - \varepsilon.
\]
by Lemma \ref{lem:free}.
Since $\varepsilon$ is arbitrary, it follows that $\| T_0 \| \leq \| T \|$.
Thus there exists a conditional expectation $F: B \to B^{(0)}$ such that
$F \circ \Psi = \Psi |_{\mathcal{O}_{X_R} ^{\mathbb{T}}} \circ E$,
where $B^{(0)}$ is the closed linear span
of $\{ \rho(a) , V_{\xi} V_{\eta} ^* \, | \, a \in A, \, \xi, \eta \in
X_R ^{\otimes k},\,  k \in \mathbb{N} \}$ and
$E$ is the conditional expectation form $\mathcal{O}_{X_R}$
onto $\mathcal{O}_{X_R} ^{\mathbb{T}}$.
Since $\Psi$ is injective on $A$, $\Psi$
is injective on $\mathcal{O}_{X_R} ^{\mathbb{T}}$ by \cite{Pi}
(see also \cite[Lemma 2.2]{KPW}).
Hence $\Psi$ is injective on $\mathcal{O}_{X_R}$
and $\Phi := \Psi ^{-1}$ is the desired isomorphism.
\end{proof}

%%%%%%%%%%%%%%%%%%%%%%%%%%%%%%%%%%%%%%%%%%%%%%%%%%%%%%%%%%%%%%%%%%%%%%%%%%

We do not identify $X_R ^{\otimes m}$ with $X_{R^{\circ m}}$ in the proof of
Proposition \ref{prop:uniqueness}. But in fact we can identify
them in the following way.
We note that
the $m$-th iteration of a finite Blaschke
product is also a finite Blaschke product.

%%%%%%%%%%%%%%%%%%%%%%%%%%%%%%%%%%%%%%%%%%%%%%%%%%%%%%%%%%%%%%%%%%%%%%%%%%

\begin{prop}
Let $R$ be a finite Blaschke product. Then there exists
an isomorphism $\Psi: X_R ^{\otimes m} \to X_{R^{\circ m}}$
as a Hilbert bimodule over $A$ such that
\[
  \Psi(\xi_1 \otimes \dots \otimes \xi_m) = \xi_1 (\xi_2 \circ R)
  (\xi_3 \circ R^{\circ 2}) \dots (\xi_m \circ R^{\circ (m-1)}).
\]
\end{prop}

%%%%%%%%%%%%%%%%%%%%%%%%%%%%%%%%%%%%%%%%%%%%%%%%%%%%%%%%%%%%%%%%%%%%%%%%%%

\begin{proof}
It is easy to show that $\Psi$ is well-defined and a bimodule homomorphism.
We show that $\Psi$ preserves the inner product.
For simplicity of the notation,
we consider the case when $m=2$.
\begin{align*}
\langle \xi_1 & \otimes \xi_2, \eta_1 \otimes \eta_2 \rangle_A (w_2)
= \langle \xi_2, \langle \xi_1, \eta_1 \rangle_A \cdot \eta_2 \rangle_A (w_2) \\
&= \sum_{w_1 \in R^{-1} (w_2)} \frac{1}{|R'(w_1)|} \ \overline{\xi_2(w_1)}
   \ (\langle \xi_1, \eta_1 \rangle_A \cdot \eta_2) (w_1) \\
&= \sum_{w_1 \in R^{-1} (w_2)} \frac{1}{|R'(w_1)|} \ \overline{\xi_2(w_1)}
   \left( \sum_{z \in R^{-1} (w_1)} \frac{1}{|R'(z)|} \ \overline{\xi_1 (z)} \eta_1(z) \right) \eta_2 (w_1) \\
&= \sum_{z \in (R^{\circ 2})^{-1} (w_2)} \frac{1}{|R'(R(z))R'(z)|}
   \ \overline{\xi_1 (z) \xi_2 (R(z))} \eta_1 (z) \eta _2 (R(z)) \\
&= \sum_{z \in (R^{\circ 2})^{-1} (w_2)} \frac{1}{|(R^{\circ 2})'(z)|}
   \ \overline{\xi_1 (z) \xi_2 (R(z))} \eta_1 (z) \eta _2 (R(z)) \\
&= \langle \Psi(\xi_1 \otimes \xi_2), \Psi(\eta_1 \otimes \eta_2) \rangle_A (w_2).
\end{align*}

Since $\Psi$ preserves the inner product, $\Psi$ is injective.
We next prove that $\Psi$ is surjective.
Since $\im \Psi$ is $*$-subalgebra of $X_{R^{\circ m}}$
and separates the two points, $\im \Psi$ is dense in $X_{R^{\circ m}}$
with respect to $\| \ \|$ by the Stone-Weierstrass Theorem.
Since $\| \ \|$ and $\| \ \|_2$ are equivalent and $\Psi$
is isometric with respect to $\| \ \|_2$, $\Psi$ is surjective.
\end{proof}

%%%%%%%%%%%%%%%%%%%%%%%%%%%%%%%%%%%%%%%%%%%%%%%%%%%%%%%%%%%%%%%%%%%%%%%%%%
%%%%%%%%%%%%%%%%%%%%%%%%%%%%%%%%%%%%%%%%%%%%%%%%%%%%%%%%%%%%%%%%%%%%%%%%%%

\section{Quotient algebra $\mathcal{OC}_R$}
\label{sect:main}

The following theorem is a generalization of Theorem \ref{thm:Jury}.
Let $R$ be a finite Blaschke product.
We show that the $C^*$-algebra $\mathcal{OC}_{R}$ is isomorphic to
the crossed product $C(\mathbb{T}) \rtimes_\alpha \mathbb{Z} / q \mathbb{Z}$
in Theorem \ref{thm:Jury} or the Cuntz-Pimsner algebra $\mathcal{O}_{X_R}$
associated to the Hilbert bimodule $X_R$ over $C(\mathbb{T})$.

\begin{thm} \label{thm:uniqueness}
Let $R$ be a finite Blaschke product.
If $R$ is a M\"{o}bius transformation of finite order $q$,
then $\mathcal{OC}_R$ is isomorphic to
$C(\mathbb{T}) \rtimes _{\alpha} \mathbb{Z} / q \mathbb{Z}$,
where $\alpha(a) = a \circ R$ for $a \in C(\mathbb{T})$.
Otherwise, there exists an isomorphism $\Phi: \mathcal{OC}_R \to
\mathcal{O}_{X_R}$ such that $\Phi(\pi(T_a)) = a$ for $a \in C(\mathbb{T})$
and $\Phi(\pi(C_R)) = S_1$, where $\pi$ is the canonical quotient map
$\mathcal{TC}_R$ to $\mathcal{OC}_R$ and $1$ is the constant map in
$X_R$ taking constant value $1$.
\end{thm}

%%%%%%%%%%%%%%%%%%%%%%%%%%%%%%%%%%%%%%%%%%%%%%%%%%%%%%%%%%%%%%%%%%%%%%%%%%

\begin{proof}
If $R$ is a M\"{o}bius transformation of finite order,
this statement is the same as Theorem \ref{thm:Jury}.
Otherwise, it follows immediately from Propositions
\ref{prop:relation} and \ref{prop:uniqueness}.
\end{proof}

%%%%%%%%%%%%%%%%%%%%%%%%%%%%%%%%%%%%%%%%%%%%%%%%%%%%%%%%%%%%%%%%%%%%%%%%%%

\begin{rem}
Let $R$ be a finite Blaschke product and
$\alpha(a) = a \circ R$ for $a \in C(\mathbb{T})$.
It is easy to see that
the Aleksandrov operator $A_R$ on $C(\mathbb{T})$
is a transfer operator for the pair $(C(\mathbb{T}), \alpha)$
in the sense of Exel \cite{E}.
If $R$ is a M\"{o}bius transformation of infinite order which maps $\mathbb{D}$ to itself 
or a finite Blaschke product of degree at least two, then the element $\pi(C_R)$ of the composition
operator in $\mathcal{OC}_R$ corresponds exactly to the implementing
element in Exel's crossed product
$C(\mathbb{T}) \rtimes_{\alpha, A_R} \mathbb{N}$ in \cite{E}.
It follows directly from the fact that $\mathcal{O}_{X_R}$ is naturally
isomorphic to $C(\mathbb{T}) \rtimes_{\alpha, A_R} \mathbb{N}$.
\end{rem}

\begin{rem}
Let $R$ be a M\"{o}bius transformation of finite order $q$.
It is known that $R$ is conjugate to $\varphi(z) = e^{2 \pi i (p/q)} z$
by a M\"{o}bius transformation of $\mathbb{D}$
where $p/q$ is in lowest terms.
By Theorem \ref{thm:uniqueness}, $\mathcal{OC}_R$ is isomorphic to
$C(\mathbb{T}) \rtimes _{\beta} \mathbb{Z} / q \mathbb{Z}$,
where $\beta(a) = a \circ \varphi$ for $a \in C(\mathbb{T})$.
It is known that the crossed product
$C(\mathbb{T}) \rtimes _{\beta} \mathbb{Z} / q \mathbb{Z}$
is isomorphic to $M_q (C(\mathbb{T}))$
since the orbit space $\mathbb{T} / \varphi$
is homeomorphic to $\mathbb{T}$.
Therefore $\mathcal{OC}_R$ is isomorphic to $M_q (C(\mathbb{T}))$.
\end{rem}

%%%%%%%%%%%%%%%%%%%%%%%%%%%%%%%%%%%%%%%%%%%%%%%%%%%%%%%%%%%%%%%%%%%%%%%%%%

\begin{cor} \label{cor:nuclear}
Let $R$ be a finite Blaschke product.
Then $\mathcal{OC}_R$ is separable and nuclear,
and satisfies the Universal Coefficient Theorem.
\end{cor}

%%%%%%%%%%%%%%%%%%%%%%%%%%%%%%%%%%%%%%%%%%%%%%%%%%%%%%%%%%%%%%%%%%%%%%%%%%

\begin{proof}
This follows from Theorem \ref{thm:uniqueness},
\cite[Corollary 7.4, Proposition 8.8]{K},
and the above remark.
\end{proof}

%%%%%%%%%%%%%%%%%%%%%%%%%%%%%%%%%%%%%%%%%%%%%%%%%%%%%%%%%%%%%%%%%%%%%%%%%%

We compute the $K$-group of $\mathcal{OC}_R$.
The finite orthonormal basis $\{u_i \}_{i=1} ^n$ of $X_R$
plays an important role in this proof.

%%%%%%%%%%%%%%%%%%%%%%%%%%%%%%%%%%%%%%%%%%%%%%%%%%%%%%%%%%%%%%%%%%%%%%%%%%

\begin{thm} \label{thm:K-group}
Let $R$ be a finite Blaschke product of degree $n$.
\begin{enumerate}
\item
If $n =1$ and $R$ has finite order,
then $K_0 (\mathcal{OC}_R) \cong \mathbb{Z}$
and $K_1 (\mathcal{OC}_R) \cong \mathbb{Z}$.
\item
If $n =1$ and $R$ has infinite order, then
$K_0 (\mathcal{OC}_R) \cong \mathbb{Z}^2$
and $K_1 (\mathcal{OC}_R) \cong \mathbb{Z}^2$.
\item
If $n \geq 2$,
then $(\, K_0 (\mathcal{OC}_R) , \,  [I]_0, \, K_1 (\mathcal{OC}_R) \,)
\cong
(\, \mathbb{Z} \oplus \mathbb{Z}/ (n-1) \mathbb{Z}, \, (0, 1), \, \mathbb{Z} \, )$.
\end{enumerate}
\end{thm}

%%%%%%%%%%%%%%%%%%%%%%%%%%%%%%%%%%%%%%%%%%%%%%%%%%%%%%%%%%%%%%%%%%%%%%%%%%

\begin{proof}
By the above remark, (1) is obvious. 
It is sufficient to compute the K-group of
$\mathcal{O}_{X_R}$ by Theorem \ref{thm:uniqueness}.
Put $A=C(\mathbb{T})$.
We use the following six-term exact sequence due to Pimsner \cite{Pi}
(see also \cite{K}).

\begin{align*}
\begin{CD}
K_0(A)  @>\text{${\rm id}-[X_R]_0$}>> K_0(A) @>\text{$i_*$}>> K_0(\mathcal{O}_{X_R})\\
@A\text{$\delta_1$}AA @. @VV\text{$\delta_0$}V \\
K_1(\mathcal{O}_{X_R})  @<<\text{$i_*$}< K_0(A) @<<\text{${\rm id}-[X_R]_1$}< K_1(A),
\end{CD}
\end{align*}
where $i: A \to \mathcal{O}_{X_R}$ is the natural inclusion and
$[X_R]_j$ are the following composition maps
\begin{align*}
\begin{CD}
K_j(A)  @>\text{$\phi_*$}>> K_j(\mathcal{L}(X_R)) @>{\text{$\cong$}}>> K_j(A)
\end{CD}
\end{align*}
for $j=0,1$.
We can identify $\mathcal{L}(X_R)$ with $M_n(A)$ using the finite orthonormal
basis $\{u_i \}_{i=1} ^n$ of $X_R$. By the definition of the action of $X_R$,
$\phi(R) u_i = u_i \cdot e_1$ for
$i= 1, \dots , n$. Thus
\begin{align*}
  \phi(1) =       \begin{pmatrix}
                    1 & 0 & \dots & 0 \\
                    0 & 1 & \dots & 0 \\
                    \vdots &\vdots & \ddots & \vdots\\
                    0 & 0 & \dots & 1
                  \end{pmatrix}
\quad
\text{and}
\quad
\phi(R) =         \begin{pmatrix}
                    e_1 & 0 & \dots & 0 \\
                    0 & e_1 & \dots & 0 \\
                    \vdots &\vdots & \ddots & \vdots\\
                    0 & 0 & \dots & e_1
                  \end{pmatrix},
\end{align*}
where $e_1(z) = z$ for $z \in \mathbb{T}$.
Since $R:\mathbb{T} \to \mathbb{T}$ is an orientation-preserving $1$ to $n$
map, we have  $[R]_1 = n[e_1]_1$.
Thus $[X_R]_0 ([1]_0) = n[1]_0$
and $[X_R]_1 (n [e_1]_1) = [X_R]_1 ([R]_1) = n[e_1]_1$.
Since $K_0(A) = \mathbb{Z} [1]_0$ and $K_1(A) = \mathbb{Z} [e_1]_1$,
$[X_R]_0$ is the endomorphism induced by multiplication
by $n$ and $[X_R]_1$ is the identity map.
Therefore we can compute the K-group by the above exact sequence.
Let $n \geq 2$. Using again the above exact sequence, $[1]_0$
corresponds to $(k,1) \in \mathbb{Z} \oplus \mathbb{Z}/ (n-1) \mathbb{Z}$
for some $k$.
Since $\sum_{i=1} ^n S_{u_i} S_{u_i} ^* = 1$
and $S_{u_i} ^*  S_{u_i} = 1$, we have $(n-1) [1]_0 = 0$.
Thus $k = 0$.
\end{proof}

%%%%%%%%%%%%%%%%%%%%%%%%%%%%%%%%%%%%%%%%%%%%%%%%%%%%%%%%%%%%%%%%%%%%%%%%%%

Let $R$ be a finite Blaschke product of degree at least two.
We next consider the simplicity of the $C^*$-algebra $\mathcal{OC}_R$.
We show that it can be characterized by the dynamics 
near the Denjoy-Wolff point of $R$.
First we consider a relation between the $C^*$-algebra $\mathcal{OC}_R$
and the $C^*$-algebra $\mathcal{O}_R (J_R)$ associated with the
complex dynamical system introduced in \cite{KW}.

Let $R$ be a rational function of degree at least two.
We recall the definition of the $C^*$-algebra
$\mathcal{O}_R (J_R)$. 
Since the Julia set $J_R$ is completely invariant under $R$, that is,
$R(J_R) = J_R = R^{-1}(J_R)$, we can consider the restriction 
$R|_{J_R} : J_R \rightarrow J_R$.
Let $B= C(J_R)$ and $X = C(\graph R|_{J_R})$,
where $\graph R|_{J_R} = \{(z,w) \in J_R \times J_R \, \, | \, \, w = R(z)\} $ 
is the graph of $R$.
We denote by $e(z)$ the branch index of $R$ at $z$.
Then $X$ is a $B$-$B$ bimodule over $B$ by 
\[
  (a\cdot f \cdot b)(z,w) = a(z)f(z,w)b(w),\quad a,b \in B, \, 
  f \in X.
\]
We define a $B$-valued inner product $\langle \ , \ \rangle_B$ on $X$ by 
\[
  \langle f, g \rangle_B(w) = \sum _{z \in R^{-1}(w)} e(z)
  \overline{f(z,w)}g(z,w),
  \quad f,g \in X, \, w \in J_R.
\]
Thanks to the branch index $e(z)$, the inner product above gives a continuous
function and $X$ is a full Hilbert bimodule over $B$ without completion. 
The left action of $B$ is unital and faithful. 
The $C^*$-algebra ${\mathcal O}_R(J_R)$
is defined as the Cuntz-Pimsner algebra of the Hilbert bimodule 
$X= C(\graph R|_{J_R})$ over 
$B = C(J_R)$. 

%%%%%%%%%%%%%%%%%%%%%%%%%%%%%%%%%%%%%%%%%%%%%%%%%%%%%%%%%%%%%%%%%%%%%%%%%%%%%

Let $R$ be a finite Blaschke product of degree at least two.
By Proposition \ref{prop:classification}, the Julia set
$J_R$ is $\mathbb{T}$ or a Cantor set on $\mathbb{T}$.
Since $R$ has no branched point on $\mathbb{T}$,
the branch index  of $R$ is $e(z) = 1$ for $z \in \mathbb{T}$.
Let $Y_R = C(J_R)$. Then $Y_R$ is a $B$-$B$ bimodule
by
\[
   (a \cdot \xi \cdot b) (z)
   = a(z) \xi(z) b(R(z)), \quad a, b \in B, \, \xi \in Y_R.
\]
We define a $B$-valued inner product $\langle \ , \ \rangle_B$ on $Y_R$ by
\[
  \langle \xi, \eta \rangle_B (w) = \sum _{z \in R^{-1} (w)} \frac{1}{|R'(z)|}
  \overline{\xi(z)} \eta(z), \quad \xi, \eta \in Y_R, \, w \in J_R.
\]
Let $\Psi:X \to Y_R$ be given by $(\Psi(f))(z)
= \sqrt{|R'(z)|} \, f(z, R(z))$ for
$f \in X$ and $z \in J_R$.
Since $|R'|$ is a positive invertible function on $J_R$,
it is easy to see that $\Psi$ is an isomorphism of Hilbert bimodules over $B$.
Therefore $\mathcal{O}_R(J_R)$ is isomorphic to the Cuntz-Pimsner algebra
$\mathcal{O}_{Y_R}$.

We shall show a
quotient algebra of $\mathcal{OC}_R$ by an ideal induced by
the Julia set of $R$ is isomorphic to the $C^*$-algebra $\mathcal{O}_R(J_R)$
associated with the complex dynamical 
system $\{ R^{\circ m} \}_{m=1} ^{\infty}$.
We use a theorem on ideals of the Cuntz-Pimsner algebras.
There have been many studies on ideals of Cuntz-Pimsner algebras
\cite{FMR, KPW, K2, Pi}.

%%%%%%%%%%%%%%%%%%%%%%%%%%%%%%%%%%%%%%%%%%%%%%%%%%%%%%%%%%%%%%%%%%%%%%%%%%%%%

\begin{prop} \label{prop:ideal}
Let $R$ be a finite Blaschke product of degree at least two
and let $\pi$ be the canonical quotient map $\mathcal{TC}_R$ to
$\mathcal{OC}_R$.
Suppose $\mathcal{J}_R$ is the ideal of $\mathcal{OC}_R$
generated by
$\{ \pi(T_a) \, | \, a \in C(\mathbb{T}), \, a|_{J_R}=0 \}$.
Then $\mathcal{OC}_R / \mathcal{J}_R$ is
isomorphic to $\mathcal{O}_R (J_R)$.
In particular, $\mathcal{OC}_R / \mathcal{J}_R$ is simple.
Moreover, if $J_R = \mathbb{T}$, then
$\mathcal{OC}_R$ is isomorphic to $\mathcal{O}_R (J_R)$.
\end{prop}

%%%%%%%%%%%%%%%%%%%%%%%%%%%%%%%%%%%%%%%%%%%%%%%%%%%%%%%%%%%%%%%%%%%%%%%%%%%%%

\begin{proof}
Let $I_R = \{ a \in A = C(\mathbb{T}) \ | \  a|_{J_R} = 0 \}$
and let $\mathcal{I}_R$ be the ideal generated
by $I_R$ in $\mathcal{O}_{X_R}$. It suffices to show that
$\mathcal{O}_{X_R} / \mathcal{I}_R$ is isomorphic to $\mathcal{O}_{Y_R}$
because
there exists an isomorphism $\Phi: \mathcal{OC}_R \to \mathcal{O}_{X_R}$
as in Theorem \ref{thm:uniqueness}.
Since $J_R$ is completely invariant under $R$, we have
$\langle \xi, a \cdot \eta \rangle_A = 0$ for $a \in I_R$.
Thus $I_R$ is an $X_R$-invariant ideal. See \cite{FMR, KPW} for the
definition of invariant ideals.
Therefore $X_R / X_R I_R$ is naturally a Hilbert bimodule over $A/I_R$
whose left action is faithful.
Since $X_R$ has a finite basis, $\phi(A)$ is included in
$\mathcal{K}(X_R) = \mathcal{L}(X_R)$,
where $\phi$ is the left action of $X_R$.
By \cite[Corollary 3.3]{FMR},
the Cuntz-Pimsner algebra $\mathcal{O}_{X_R / X_R I_R}$
is canonically isomorphic to $\mathcal{O}_{X_R} / \mathcal{I}_R$.
We can identify $X_R / X_R I_R$ with the Hilbert bimodule $Y_R$ over
$B \cong A/I_R$, which completes the proof.
\end{proof}

%%%%%%%%%%%%%%%%%%%%%%%%%%%%%%%%%%%%%%%%%%%%%%%%%%%%%%%%%%%%%%%%%%%%%%%%%%%%%

We give a characterization of simplicity of the $C^*$-algebra
$\mathcal{OC}_R$ as a corollary using the Julia set of $R$.

\begin{cor} \label{cor:kirchberg}
Let $R$ be a finite Blaschke product of degree at least two.
Then the $C^*$-algebra
$\mathcal{OC}_R$ is simple if and only if $J_R =\mathbb{T}$.
Furthermore, if $J_R = \mathbb{T}$, then $\mathcal{OC}_R$ is
purely infinite.
\end{cor}

%%%%%%%%%%%%%%%%%%%%%%%%%%%%%%%%%%%%%%%%%%%%%%%%%%%%%%%%%%%%%%%%%%%%%%%%%%%%%

\begin{proof}
We use the same notation in the previous proof.
Suppose  $J_R \neq \mathbb{T}$. Then $I_R$ is a proper ideal in
$A=C(\mathbb{T})$.
The isomorphism $\mathcal{O}_{X_R / X_R I_R}$
to $\mathcal{O}_{X_R} / \mathcal{I}_R$ in the previous proof
maps $[a]_{A/I_R}$ to
$[a]_{\mathcal{O}_{X_R} / \mathcal{I}_R}$ for $a \in A$,
where $[a]_{A/I_R}$ and $[a]_{\mathcal{O}_{X_R} / \mathcal{I}_R}$ are
images of $a \in A$ under quotient maps $A \to A/I_R$ and
$\mathcal{O}_{X_R} \to \mathcal{O}_{X_R} / \mathcal{I}_R$ respectively.
Therefore $\mathcal{I}_R$
is a proper ideal of $\mathcal{O}_{X_R}$
and $\mathcal{OC}_R$ is not simple.
The rest follows immediately from Proposition \ref{prop:ideal} and
\cite[Theorem 3.8]{KW}.
\end{proof}

%%%%%%%%%%%%%%%%%%%%%%%%%%%%%%%%%%%%%%%%%%%%%%%%%%%%%%%%%%%%%%%%%%%%%%%%%%%%%

We show that the simplicity of the $C^*$-algebra $\mathcal{OC}_R$
can be characterized by the dynamics near the
Denjoy-Wolff point of $R$.

%%%%%%%%%%%%%%%%%%%%%%%%%%%%%%%%%%%%%%%%%%%%%%%%%%%%%%%%%%%%%%%%%%%%%%%%%%%%%

\begin{thm} \label{thm:simple}
Let $R$ be a finite Blaschke product of degree at least two.
If $R$ has a fixed point in $\mathbb{D}$ or
a parabolic fixed point on $\mathbb{T}$ with two attracting petals,
then the $C^*$-algebra $\mathcal{OC}_R$ is simple.
If $R$ has an attracting fixed point on $\mathbb{T}$ or
a parabolic fixed point on $\mathbb{T}$ with one attracting petal,
then $C^*$-algebra $\mathcal{OC}_R$ is not simple.
\end{thm}

%%%%%%%%%%%%%%%%%%%%%%%%%%%%%%%%%%%%%%%%%%%%%%%%%%%%%%%%%%%%%%%%%%%%%%%%%%%%%

\begin{proof}
Use Propositions \ref{prop:classification}
and \ref{prop:ideal}.
\end{proof}

%%%%%%%%%%%%%%%%%%%%%%%%%%%%%%%%%%%%%%%%%%%%%%%%%%%%%%%%%%%%%%%%%%%%%%%%%%%%%

\begin{ex}
Let $P_n(z) = z^n$ for $n \geq 2$.
Since $R$ has a fixed point $0$ in $\mathbb{D}$,
the Julia set $J_{P_n}$ is $\mathbb{T}$.
Therefore the $C^*$-algebra $\mathcal{OC}_{P_n}$ is simple.
We note that
the finite orthonormal basis $\{u_k \}_{k=1} ^n$ of $X_{P_n}$ is given by
$u_k (z) = z^{k-1}$ for $k = 1, \dots ,n$.
\end{ex}
%%%%%%%%%%%%%%%%%%%%%%%%%%%%%%%%%%%%%%%%%%%%%%%%%%%%%%%%%%%%%%%%%%%%%%%%%%%%%

We gave the following examples $R_1, \dots , R_4$ in the introduction of this
papar. Now we discuss the simplicity of these examples using Theorem
\ref{thm:simple}.

%%%%%%%%%%%%%%%%%%%%%%%%%%%%%%%%%%%%%%%%%%%%%%%%%%%%%%%%%%%%%%%%%%%%%%%%%%%%%

\begin{ex}
Let $R_1(z)= \frac{2z^2-1}{2-z^2}$.
Since $R_1$ has a fixed point $\frac{-3 + \sqrt{5}}{2}$ in
$\mathbb{D}$, the Julia set $J_{R_1}$ is $\mathbb{T}$.
Therefore the $C^*$-algebra $\mathcal{OC}_{R_1}$ is simple.
\end{ex}

%%%%%%%%%%%%%%%%%%%%%%%%%%%%%%%%%%%%%%%%%%%%%%%%%%%%%%%%%%%%%%%%%%%%%%%%%%%%%

\begin{ex}
Let $R_2(z)= \frac{2z^2+1}{2+z^2}$.
Since $R_2$ has an attracting fixed point $1$ on
$\mathbb{T}$, the Julia set $J_{R_2}$ is a Cantor set on $\mathbb{T}$.
Therefore the $C^*$-algebra $\mathcal{OC}_{R_2}$ is not simple.
\end{ex}

%%%%%%%%%%%%%%%%%%%%%%%%%%%%%%%%%%%%%%%%%%%%%%%%%%%%%%%%%%%%%%%%%%%%%%%%%%%%%

\begin{ex}
Let $R_3(z)= \frac{3z^2+1}{3+z^2}$.
Since $R_3$ has a parabolic fixed point $1$ on
$\mathbb{T}$ with two attracting petals, the Julia set $J_{R_3}$
is $\mathbb{T}$.
Therefore the $C^*$-algebra $\mathcal{OC}_{R_3}$ is simple.
\end{ex}

%%%%%%%%%%%%%%%%%%%%%%%%%%%%%%%%%%%%%%%%%%%%%%%%%%%%%%%%%%%%%%%%%%%%%%%%%%%%%

\begin{ex}
Let $R_4(z)= \frac{(3+i)z^2+(1-i)}{(3-i)+(1+i)z^2}$.
It is easy to see that $R_4$ is a finite Blaschke product.
Since $R_4$ has a parabolic fixed point $1$ on
$\mathbb{T}$ with one attracting petal,
the Julia set $J_{R_4}$ is a Cantor set on $\mathbb{T}$.
Therefore the $C^*$-algebra $\mathcal{OC}_{R_4}$ is not simple.
\end{ex}

%%%%%%%%%%%%%%%%%%%%%%%%%%%%%%%%%%%%%%%%%%%%%%%%%%%%%%%%%%%%%%%%%%%%%%%%%%%%%

The following theorem implies that
the degree is
a complete isomorphism invariant for the class of $\mathcal{OC}_R$
such that $R$ is a finite Blaschke product of degree
at least two and $J_R = \mathbb{T}$.

%%%%%%%%%%%%%%%%%%%%%%%%%%%%%%%%%%%%%%%%%%%%%%%%%%%%%%%%%%%%%%%%%%%%%%%%%%%%%

\begin{thm} \label{thm:classify}
Let $R_1$ and $R_2$ be finite Blaschke products of degree at least two.
If $\mathcal{OC}_{R_1}$ is isomorphic to $\mathcal{OC}_{R_2}$, then
$R_1$ and $R_2$ have the same degree.
Moreover, suppose $J_{R_1} = J_{R_2} = \mathbb{T}$.
If $R_1$ and $R_2$ have the same degree, then
$\mathcal{OC}_{R_1}$ is isomorphic to $\mathcal{OC}_{R_2}$.
\end{thm}

%%%%%%%%%%%%%%%%%%%%%%%%%%%%%%%%%%%%%%%%%%%%%%%%%%%%%%%%%%%%%%%%%%%%%%%%%%%%%

\begin{proof}
This follows immediately from Corollary \ref{cor:nuclear}, Theorem
\ref{thm:K-group}, Corollary \ref{cor:kirchberg} and \cite{Ki,Ph}.
\end{proof}

%%%%%%%%%%%%%%%%%%%%%%%%%%%%%%%%%%%%%%%%%%%%%%%%%%%%%%%%%%%%%%%%%%%%%%%%%%%%%

\begin{ex}
Let $R$ be a finite Blaschke product of degree $n$
at least two with $R(0) = 0$.
This case was studied in \cite{HW}.
Since $R$ has a fixed point $0$ in $\mathbb{D}$,
the Julia set $J_R$ is $\mathbb{T}$.
By Theorem \ref{thm:classify}, the $C^*$-algebra $\mathcal{OC}_R$ is
isomorphic to the $C^*$-algebra $\mathcal{OC}_{P_n}$.
This was already proved in \cite{HW},
using the fact that $R$ is topologically conjugate to $P_n$.
\end{ex}

%%%%%%%%%%%%%%%%%%%%%%%%%%%%%%%%%%%%%%%%%%%%%%%%%%%%%%%%%%%%%%%%%%%%%%%%%%%%%

\begin{ex}
Let $R_1(z)= \frac{2z^2-1}{2-z^2}$ and $R_3(z)= \frac{3z^2+1}{3+z^2}$.
Then $J_{R_1} = J_{R_3} = \mathbb{T}$ and $R_1$ and $R_3$ have the same degree
two. By Theorem \ref{thm:classify},
the $C^*$-algebra $\mathcal{OC}_{R_1}$ is isomorphic to the $C^*$-algebra
$\mathcal{OC}_{R_3}$. Moreover
$\mathcal{OC}_{R_i}$ is also isomorphic to $\mathcal{OC}_{P_2}$ for $i=1, 3$.
\end{ex}

%%%%%%%%%%%%%%%%%%%%%%%%%%%%%%%%%%%%%%%%%%%%%%%%%%%%%%%%%%%%%%%%%%%%%%%%%%%%%

Let $R$ be a finite Blaschke product of degree $n$ at least two
satisfying (2) in Proposition \ref{prop:classification}.
By Proposition \ref{prop:ideal}, the quotient algebra
$\mathcal{OC}_R / \mathcal{J}_R$ is simple. Furthermore, we
can show that $\mathcal{OC}_R / \mathcal{J}_R$ is isomorphic to
the Cuntz algebra $\mathcal{O}_n$.

%%%%%%%%%%%%%%%%%%%%%%%%%%%%%%%%%%%%%%%%%%%%%%%%%%%%%%%%%%%%%%%%%%%%%%%%%%%%%

\begin{prop}
Let $R$ be a finite Blaschke product of degree $n$ at least two
which has an attracting fixed point on $\mathbb{T}$ and
let $\pi$ be the canonical quotient map $\mathcal{TC}_R$ to
$\mathcal{OC}_R$.
Suppose $\mathcal{J}_R$ is the ideal of $\mathcal{OC}_R$
generated by
$\{ \pi(T_a) \, | \, a \in C(\mathbb{T}), \, a|_{J_R}=0 \}$.
Then $\mathcal{OC}_R / \mathcal{J}_R$ is isomorphic to the Cuntz algebra
$\mathcal{O}_n$.
\end{prop}

%%%%%%%%%%%%%%%%%%%%%%%%%%%%%%%%%%%%%%%%%%%%%%%%%%%%%%%%%%%%%%%%%%%%%%%%%%%%%

\begin{proof}
Let $w_0$ be the attracting fixed point on $\mathbb{T}$ and
let $B(R)$ be the set of all branched points of $R$.
Since $B(R) \subset \widehat{\mathbb{C}} \smallsetminus \mathbb{T}$
and $F_R$ is a connected set containing
$\widehat{\mathbb{C}} \smallsetminus \mathbb{T}$ and $w_0$,
$B(R)$ lies in the immediate attracting basin of $w_0$.
Applying Proposition \ref{prop:ideal} and \cite[Example 4.4]{KW},
we have the desired conclusion. 
\end{proof}

%%%%%%%%%%%%%%%%%%%%%%%%%%%%%%%%%%%%%%%%%%%%%%%%%%%%%%%%%%%%%%%%%%%%%%%%%%%%%

\begin{ex}
Let $R_2(z)= \frac{2z^2+1}{2+z^2}$.
Then $\mathcal{OC}_{R_2} / \mathcal{J}_{R_2}$
is isomorphic to the Cuntz algebra
$\mathcal{O}_2$.
\end{ex}

%%%%%%%%%%%%%%%%%%%%%%%%%%%%%%%%%%%%%%%%%%%%%%%%%%%%%%%%%%%%%%%%%%%%%%%%%%%%%

\begin{ex}
Let $R_0(z) = \frac{2z^2 -1}{z}$. This is considered in \cite[Example 4.4]{KW}.
Since $R_0$ maps the upper half plane to itself and $R_0(\mathbb{R}) \subset
\mathbb{R}$, the rational function $R_0$ is conjugate to a finite Blaschke
product $R$ by the Cayley transformation. Since
$R$ has an attracting fixed point $1$,
$\mathcal{OC}_R / \mathcal{J}_R$ is isomorphic to the Cuntz algebra
$\mathcal{O}_2$.
\end{ex}

%%%%%%%%%%%%%%%%%%%%%%%%%%%%%%%%%%%%%%%%%%%%%%%%%%%%%%%%%%%%%%%%%%%%%%%%%%%%%

\begin{ack}
The auther thanks Professor Yasuo Watatani for his constant encouragement
and advice.
\end{ack}

\end{document}